\documentclass[12pt, reqno]{amsart}
\usepackage{hyperref}
\usepackage{graphicx}
\usepackage{amsmath}
\usepackage{amsthm}
\usepackage{amssymb,bbm}%
\usepackage[numbers, square]{natbib}
\usepackage{color}

\usepackage[left=2.3cm, top=2.3cm,bottom=2.3cm,right=2.3cm]{geometry}

\newtheorem {theorem}{Theorem}[section]

\newtheorem {lemma}[theorem]{Lemma}
\newtheorem {corollary}[theorem]{Corollary}

{\theoremstyle{definition}
\newtheorem{definition}[theorem]{Definition}
\newtheorem {example}[theorem]{Example}
}

{
\theoremstyle{remark}
\newtheorem {remark}[theorem]{Remark}
}

\newcommand{\eqfdd}{\stackrel{\rm f.d.d.}{=}}

\newcommand{\Cov}{\operatorname{Cov}}

\newcommand{\Var}{\operatorname{Var}}

\DeclareMathOperator{\sech}{sech}

\renewcommand{\Im}{\operatorname{Im}}  

\def\ba{\begin{array}}
\def\ea{\end{array}}
\def\bea{\begin{eqnarray} \label}
\def\eea{\end{eqnarray}}
\def\be{\begin{equation} \label}
\def\ee{\end{equation}}
\def\bit{\begin{itemize}}
\def\eit{\end{itemize}}
\def\ben{\begin{enumerate}}
\def\een{\end{enumerate}}

\def\E{\mathbb{E}}

\def\P{\mathbb{P}}

\def\R{\mathbb{R}}
\def\C{\mathbb{C}}
\def\RRd1{\mathbb{R}^{d+1}}
\def\SS{\mathbb{S}}

\def\cF{\mathcal{F}}

\newcommand{\eee}{{\rm e}}

\newcommand{\eps}{\varepsilon}

\newcommand{\dd}{{\rm d}}

\begin{document}

\title[An infinite-dimensional helix]{An infinite-dimensional helix invariant under spherical projections}

\author{Zakhar Kabluchko}
\address{Zakhar Kabluchko: Institut f\"ur Mathematische Stochastik,
Westf\"alische Wilhelms-Universit\"at M\"unster,
Orl\'eans--Ring 10,
48149 M\"unster, Germany}
\email{zakhar.kabluchko@uni-muenster.de}

\date{}

\begin{abstract}
We classify all subsets $S$ of the projective Hilbert space with the following property: for every point $\pm s_0\in S$, the spherical projection of $S\backslash\{\pm s_0\}$ to the hyperplane orthogonal to $\pm s_0$ is isometric to $S\backslash\{\pm s_0\}$.
In probabilistic terms, this means that we characterize all zero-mean Gaussian processes $Z=(Z(t))_{t\in T}$  with the property that for every $s_0\in T$ the conditional distribution of $(Z(t))_{t\in T}$ given that $Z(s_0)=0$ coincides with the distribution of $(\varphi(t; s_0) Z(t))_{t\in T}$ for some function $\varphi(t;s_0)$. A basic example of such process is the stationary zero-mean Gaussian process $(X(t))_{t\in\R}$ with covariance function $\mathbb E [X(s) X(t)] = 1/\cosh (t-s)$. We show that, in general,  the process $Z$ can be decomposed into a union of mutually independent processes of two types: (i) processes of the form $(a(t) X(\psi(t)))_{t\in T}$, with $a: T\to \R$, $\psi(t): T\to \R$, and (ii) certain exceptional Gaussian processes defined on four-point index sets. The above problem is reduced to the classification of metric spaces in which in every triangle the largest side equals the sum of the remaining two sides.
\end{abstract}



\maketitle

\section{Introduction and main results}

\subsection{Introduction}
In the present paper, we shall be interested in the stationary Gaussian process $X = (X(t))_{t\in\R}$ with zero mean and covariance function
\begin{equation}\label{eq:def_covar}
\E [X(s) X(t)] = \frac 1 {\cosh (t-s)}, \qquad s,t\in\R.
\end{equation}
This process appeared in the literature~\cite{peres_virag,matsumoto_shirai,hough_book,dembo_etal,flasche_kabluchko,poplavskyi_schehr,iksanov_kabluchko} mostly in form of various time-changes. To define the time-changed processes, let $\xi_0,\xi_1,\ldots$ be i.i.d.\ standard Gaussian random variables and $(W(u))_{u\geq 0}$ a standard Brownian motion. The \textit{random Taylor series}
$$
f(t):= \sum_{k=0}^\infty \xi_k t^k, \qquad t\in (-1,1),
$$
and the \textit{random Laplace transform}
$$
g(t) := \int_{0}^\infty \eee^{-tu} \dd W(u), \qquad t>0,
$$
are zero-mean Gaussian processes characterized by their covariance functions
$$
\E [f(s) f(t)] = \frac 1 {1-st}
\quad
\text{ and }
\quad
\E [g(s) g(t)] = \frac {1} {s+t}.
$$
By comparing the covariance functions it is easy to check  that both processes are essentially time-changes of $X$, namely
\begin{equation}\label{eq:time_change}
\left(\frac{f(\tanh t)}{\cosh t}\right)_{t\in \R} \eqfdd (X(t))_{t\in\R},
\qquad
\left( \sqrt {2\eee^{2t}}  g(\eee^{2t}) \right)_{t\in \R} \eqfdd (X(t))_{t\in\R},
\end{equation}
where $\eqfdd$ denotes the equality of finite-dimensional distributions.

If $Z=(Z(t))_{t\in T}$ denotes any of the processes $X$, $f$, $g$ introduced above, then the following remarkable property holds:
\begin{quote}
For every $s_0\in T$, the conditional distribution of $(Z(t))_{t\in T}$ given that $Z(s_0)=0$ coincides with the distribution of $(\varphi(t; s_0) Z(t))_{t\in T}$ for a suitable function $\varphi(t;s_0)$.
\end{quote}
So, the law of the conditioned process is the same as the law of the original process up to multiplication by some function.
Specifically, in the case of the process $X$, for every $s_0\in\R$, the law of $(X(t))_{t\in\R}$ conditioned on $X(s_0) = 0$ is the same as the law of the process
$$
\left(\frac{\sinh (t-s_0)}{\cosh (t-s_0)}\, X(t) \right)_{t\in\R}.
$$
Moreover, for every pairwise different $s_1,\ldots, s_d\in\R$, the law of the process $(X(t))_{t\in\R}$ conditioned on $X(s_1)=\ldots=X(s_d) =0$ is the same as the law of
$$
\left(\frac{\sinh (t-s_1)}{\cosh (t-s_1)}\ldots \frac{\sinh (t-s_d)}{\cosh (t-s_d)} \,X(t) \right)_{t\in\R}.
$$

The above property has been first observed by Peres and Vir\'ag~\cite[Proposition~12]{peres_virag} for a modification of $g(t)$ in which the $\xi_k$'s are complex-valued standard normal  and was an important step in their proof that the complex zeroes of this process form a determinantal point process. The same result can be found in~\cite[Proposition~5.1.3]{hough_book}.  For the  process $g(t)$ itself, a similar property was used by~\citet[Lemma~4.2]{matsumoto_shirai} to establish the Pfaffian character of both real and complex zeroes of $g(t)$. Recently, \citet{poplavskyi_schehr} used the Pfaffian character of the zeroes of $X$  to compute the persistence exponent of $X$ and several related processes.

The aim of the present paper is to classify all Gaussian processes having the above property. In the spirit of the work of Kolmogorov~\cite{kolmogorov1,kolmogorov2}, we shall state the problem in purely geometric terms.  Namely, we regard  $(X(t))_{t\in\R}$ as a curve (a ``helix'') in the unit sphere of the Hilbert space $L^2(\Omega,\cF, \P)$, where $(\Omega, \cF, \P)$ is the probability space on which $(X(t))_{t\in\R}$ is defined. Conditioning on $X(s_0)=0$ corresponds to the orthogonal projection onto the hyperplane orthogonal to $X(s_0)$.  Because of the appearance of the function $\varphi(t;s_0)$ in the above property, it is natural to pass to the projective Hilbert space and to replace orthogonal projections by the so-called spherical projections. We are led to the problem of classifying all subsets of the projective Hilbert space that do not change their isometry type under spherical projections.


\subsection{Geometric result}
Let $H$ be a Hilbert space. The unit sphere of $H$ will be denoted by $\SS(H):= \{x\in H \colon \|x\| = 1\}$. The projective (or elliptic space) $\P(H):= \SS(H)/\pm $ is obtained from $\SS(H)$ by identifying the antipodal points $+x$ and $-x$, for all $x\in \SS(H)$. The elements of $\P(H)$ will be denoted by $\pm x, \pm y$, and so on. The projective space is endowed with the geodesic metric
$$
\rho(\pm x,  \pm y) = \arccos |\langle x, y\rangle|.
$$
For every vector $x_0\in \SS(H)$ we denote its orthogonal complement by
$$
x_0^{\bot} = \{x\in H\colon \langle x, x_0\rangle = 0\}.
$$
Let $\P(x_0^\bot)$ be the projective space constructed from the Hilbert space $x_0^\bot$.  Given an element $\pm x_0 \in \P(H)$, we define the \textit{spherical projection} $p_{\pm x_0}: \P(H)\backslash\{\pm x_0\} \to \P(x_0^\bot)$ by
\begin{equation}\label{eq:proj_def}
p_{\pm x_0}(\pm y) := \pm\frac{y-x_0 \langle x_0,y\rangle}{\sqrt{1-\langle x_0,y\rangle^2}}.
\end{equation}
In words, we first orthogonally project $\pm y$ to the hyperplane $x_0^\bot$ and then rescale the result to have unit length. Equivalently, $p_{\pm x_0}(\pm y)$ is the point in the projective space of the hyperplane $x_0^\bot$ minimizing the distance to $\pm y$.  Note that the projection is not defined for $y=x_0$.

\begin{definition}\label{def:main}
We say that a set of points $S\subset \P(H)$ \textit{does not change its isometry type under spherical projections} if for all $\pm s_0 \in S$ and all $\pm x,\pm y\in S\backslash\{\pm s_0\}$ we have
$$
\rho(p_{\pm s_0}(\pm x),  p_{\pm s_0}(\pm y)) = \rho(\pm x,\pm y).
$$
\end{definition}
\begin{remark}
By definition of the metric $d$, the above can be written as
\begin{equation}\label{eq:scalar_prod_preserved}
|\langle p_{\pm s_0}(\pm x),  p_{\pm s_0}(\pm y) \rangle| = |\langle x,y\rangle|.
\end{equation}
\end{remark}
Our aim is to describe all sets $S$ having this property, up to isometry. Let us first consider some examples.

\begin{example}[The helix]\label{ex:helix}
Consider a ``helix'' $\{h(t)\}_{t\in \R}$ in an infinite-dimensional projective Hilbert space $\P(H)$ with the property
$$
|\langle h(s), h(t)\rangle| = \frac{1}{\cosh (t-s)}, \qquad s,t\in\R.
$$
For example, we can take $H:= L^2(\Omega,\cF,\P)$ to be the $L^2$-space of the probability space on which the Gaussian process $(X(t))_{t\in \R}$ with covariance function~\eqref{eq:def_covar} is defined, and then put $h(t) := \pm X(t)\in \P(H)$. Then, the set $S = \{h(t)\}_{t\in \R}$ satisfies the condition from Definition~\ref{def:main}. To see this, observe that for every $s_0, x\in \R$ with $x\neq s_0$ we have
$$
p_{h(s_0)}(h(x))
=
\pm \frac {X(x)- X(s_0)/\cosh (x-s_0)}{\sqrt{1-1/\cosh^2(x-s_0)}}
=
\pm \frac {X(x) \cosh (x-s_0) - X(s_0)}{\sinh (x-s_0)}.
$$
Given this, one easily checks that for every $x\neq s_0$ and $y\neq s_0$,
$$
|\langle p_{h(s_0)}(h(x)),  p_{h(s_0)}(h(y)) \rangle| = \frac 1 {\cosh (x-y)} = |\langle h(x),h(y)\rangle|.
$$
Trivially, any subset of $S$ also satisfies the condition from Definition~\ref{def:main}.
\end{example}

\begin{example}[Orthogonal unions]
If $S_\alpha\subset \P(H)$, $\alpha\in I$,  are mutually \textit{orthogonal} sets such that each $S_\alpha$ satisfies the condition from Definition~\ref{def:main}, then one easily checks that their union $\cup_{\alpha\in I} S_\alpha$ also satisfies this condition. Orthogonality means that $\langle u,v \rangle=0$ for all $\pm u\in S_\alpha$ and $\pm v\in S_\beta$ with $\alpha\neq \beta$.
\end{example}

\begin{example}[A family of exceptional quadruples]\label{ex:four_points_projective}
Let $A,B,C,D$ be four points in the unit sphere $\SS(H)$ with the following scalar products
$$
\langle A,B\rangle = \langle C,D\rangle = \frac 1 {\cosh x},
\quad
\langle A, D\rangle = \langle B,C\rangle = \frac 1 {\cosh y},
\quad
\langle A,C\rangle = \langle B,D\rangle = \frac 1 {\cosh (x-y)}.
$$
Here, $x>0$ and $y>0$ are distinct numbers with the property that the Gram matrix of $A,B,C,D$ is positive semi-definite. The eigenvalues of the Gram matrix are given by
\begin{align*}
\lambda_1 &= 1 + \sech y  +\sech(x - y) + \sech x,\\
\lambda_2 &= 1 + \sech y  -\sech(x - y) - \sech x,\\
\lambda_3 &= 1 - \sech y  +\sech(x - y) - \sech x,\\
\lambda_4 &= 1 - \sech y  -\sech(x - y) + \sech x.
\end{align*}
These formulae can be proved by comparing the characteristic polynomial of the Gram matrix with the polynomial $\prod_{k=1}^4 (\lambda-\lambda_k)$. The Gram matrix is positive semi-definite iff all $\lambda_k$'s are non-negative. We always have $\lambda_1>0$. The set of admissible pairs $(x,y)$, i.e.\ pairs  for which the remaining eigenvalues are non-negative and $x\neq y$, is shown on Figure~\ref{fig}.

\begin{figure}[h]
\begin{center}
\includegraphics[width=0.4\textwidth]{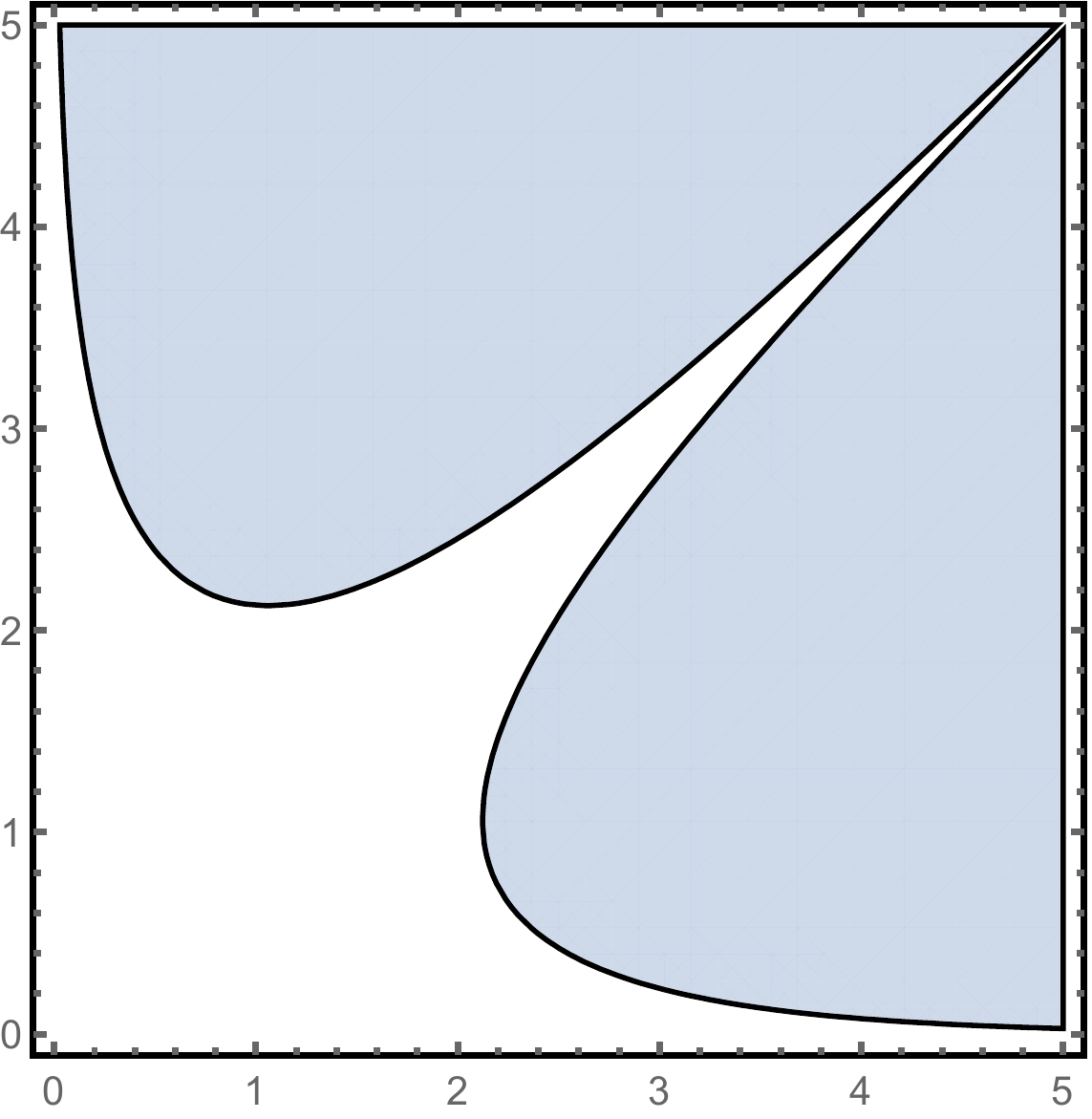}
\caption
{\small The set of admissible pairs $(x,y)$.}
\label{fig}
\end{center}
\end{figure}

We now claim that the set $S= \{\pm A, \pm B, \pm C, \pm D\}\subset \P(H)$  satisfies the condition of Definition~\ref{def:main}. To see this, it suffices to check the condition for $s_0=\pm A$ (the rest follows by symmetry reasons). We have
$$
p_{\pm A} (\pm B) = \pm \frac{B \cosh x - A}{\sinh x},
\;
p_{\pm A} (\pm C) = \pm \frac{C\cosh (x-y) - A}{\sinh (x-y)},
\;
p_{\pm A} (\pm D) = \pm \frac{D \cosh y - A}{\sinh y}.
$$
Hence,
$$
\langle p_{\pm A} (\pm B), p_{\pm A} (\pm C) \rangle
=
\pm \frac{\cosh (x-y) \cosh x/ \cosh y + 1 - 1 - 1}{\sinh x \sinh (x-y)} = \pm \frac 1 {\cosh y}
=
\pm \langle B, C \rangle.
$$
The relations for the pairs $C,D$ and $D,B$ can be checked similarly.

Finally, we claim that $S=\{\pm A, \pm B,\pm C,\pm D\}$ is not isometric to a subset of the helix from Example~\ref{ex:helix}. Our conditions on $x$ and $y$ ensure that the points $\pm A,\ldots,\pm D$ are pairwise different. If $t_1<t_2<t_3<t_4$ are real numbers, then $|\langle h(t_1), h(t_4) \rangle|$ is strictly smaller than the remaining scalar products $|\langle h(t_i), h(t_j)\rangle|$ with  $1\leq i<j\leq 4$, $(i,j)\neq (1,4)$. On the contrary, in the set $S$ all pairwise scalar products  can be decomposed into $3$ groups each consisting of $2$ equal products.
\end{example}

Now we can state our main result classifying sets which do not change their isometry type under spherical projections.
\begin{theorem}\label{theo:main1}
Let $S\subset \P(H)$ be a set satisfying the condition of Definition~\ref{def:main}. Then, we can represent $S$ as a disjoint union $S= \cup_{\alpha \in I} S_\alpha$ of pairwise orthogonal sets $S_\alpha$, $\alpha\in I$,  such that each $S_\alpha$ is  isometric either to a subset of the helix from Example~\ref{ex:helix} or to a four-point configuration from Example~\ref{ex:four_points_projective}.
\end{theorem}

We shall prove Theorem~\ref{theo:main1} and its corollaries  in Section~\ref{sec:proof_main1}. The classification of Theorem~\ref{theo:main1} simplifies considerably  if we restrict our attention to sets $S$ which are continuous curves.

\begin{corollary}\label{cor:curve}
Let $\gamma: \R\to \P(H)$ be a continuous, injective map such that its image $\gamma(\R)$ satisfies the condition of Definition~\ref{def:main}. Then, $\gamma(\R)$ is isometric to $\{h(t)\colon t\in J\}$, where $h$ is as in Example~\ref{ex:helix} and  $J\subset \R$ is an open interval (which may be bounded, half-infinite or equal to $\R$).
\end{corollary}

Let us finally restate Theorem~\ref{theo:main1} in the language of Gaussian processes.

\begin{corollary}\label{cor:gaussian}
Consider a  zero-mean Gaussian process $Z=(Z(t))_{t\in T}$ such that for every $s_0\in T$ the conditional distribution of $(Z(t))_{t\in T}$ given that $Z(s_0)=0$ coincides with the distribution of $(\varphi(t; s_0) Z(t))_{t\in T}$ for some function $\varphi(t;s_0)$. Then, there is a disjoint decomposition  $T = \cup_{\alpha\in I} T_\alpha$ and a function $a: T\to \R$  such that the following hold:
\begin{itemize}
\item[(i)] $(Z(t))_{t\in T_\alpha}$ is independent of $(Z(t))_{t\in T_\beta}$ for all $\alpha\neq \beta$;
\item[(ii)] for each $\alpha\in I$ either there is a function $\psi_\alpha: T_\alpha \to \R$ such that
$$
(Z(t))_{t\in T_\alpha} \eqfdd (a(t) X(\psi_\alpha(t)))_{t\in T_\alpha}
$$
or there is a function $\psi_\alpha: T\to \{A,B,C,D\}$ such that
$$
(Z(t))_{t\in T_\alpha} \eqfdd (a(t) Y(\psi_\alpha(t)))_{t\in T_\alpha},
$$
where $(X(s))_{s\in\R}$ is as in~\eqref{eq:def_covar},  and $Y_{x,y} = ((Y(s))_{s\in \{A,B,C,D\}}$ is a zero-mean, unit variance Gaussian process with
$
\E [Y(A) Y(B)] = \E [Y(C) Y(D)]  = 1/ {\cosh x}
$,
$
\E [Y(A) Y(D)] = \E [Y(B) Y(C)] = 1/ {\cosh y},
$
$
\E [Y(A) Y(C)] = \E [Y(B) Y(D)] = 1/ {\cosh (x-y)}
$
for some pair $(x,y)$ which is admissible in the sense of Example~\ref{ex:four_points_projective}.
\end{itemize}
\end{corollary}

\subsection{Metric spaces with triangle equality}
In the proof of Theorem~\ref{theo:main1} given in Section~\ref{sec:proof_main1} we shall reduce Theorem~\ref{theo:main1} to the classification of metric spaces with the following property.
\begin{definition}
We say that a metric space $(E,d)$ satisfies the \textit{triangle equality} if for every three points $x,y,z\in E$ the largest of the numbers $d(x,y)$, $d(y,z)$, $d(z,x)$ equals the sum of the remaining two.
\end{definition}

An example of such metric space is any subset of the real line with the usual metric $d(x,y)= |x-y|$.

\begin{example}\label{example:4_points}
The following space of $4$ points satisfies the triangle equality but cannot be isometrically embedded into the real line: $E= \{A,B,C,D\}$ with
$$
d(A,B) = d(C,D) = x, \quad
d(A, D) = d(B,C) = y, \quad
d(A,C) = d(B,D) = |x-y|.
$$
Here, $x>0$ and $y>0$ are arbitrary numbers with $x\neq y$.
\end{example}

\begin{theorem}\label{theo:main2}
If $(E,d)$ is a metric space satisfying the triangle equality and whose cardinality is different from $4$, then it is isometric to a subset of the real line. If $E$ has exactly $4$ points, then it either can be isometrically embedded into the real line or is isometric to one of the spaces from Example~\ref{example:4_points}.
\end{theorem}
The proof will be given in Section~\ref{proof:main_2}.

\subsection{Open questions}
The property of Gaussian processes studied here was used in~\cite{peres_virag,matsumoto_shirai,poplavskyi_schehr} to establish the determinantal/Pfaffian character of the zeroes of the corresponding process. It is natural to ask for a description of all (sufficiently smooth) stationary, centered, Gaussian processes whose zeroes form a Pfaffian/determinantal point process.
For example, the zero-mean, stationary \textit{complex-valued} Gaussian process $(X_{\C}(t))_{t\in\R}$ with
$$
\E [X_{\C}(s) X_{\C}(t)] = 0, \qquad \E [X_{\C}(s) \overline{X_{\C}(t)}] = \frac{1}{\cosh (s-\overline{t})}, \qquad s,t\in\R,
$$
can be extended to an analytic function on the strip $\{t\in \C\colon |\Im t| <\pi/4\}$ and its complex zeroes form a determinantal point process with  kernel
$$
K(s,t) = \frac{1}{\cosh^2 (s-\overline{t})}.
$$
This can be easily derived from the result of~\cite{peres_virag} by applying the time-change~\eqref{eq:time_change}. It is natural to conjecture that if a zero-mean, unit-variance, stationary complex Gaussian process admits an analytic continuation to some strip $\{t\in \C: |\Im t|<\eps\}$ and its zeroes form a determinantal point process there, then this process has the same law as $(\eee^{i\kappa t} X_\C(\alpha t))_{t\in\R}$ for some $\kappa\in\R$ and $\alpha>0$. Similarly, one may wonder whether every stationary, smooth, zero-mean and unit-variance Gaussian process on $\R$ whose real zeroes form a Pfaffian point process is necessarily of the form $(X(\alpha t))_{t\in\R}$ for some $\alpha>0$.


\section{Proof of Theorem~\ref{theo:main1} and its corollaries}\label{sec:proof_main1}

\subsection{Proof of Theorem~\ref{theo:main1}}
Let $S\subset \P(H)$ be a set having the property of Definition~\ref{def:main}.
\begin{lemma}\label{lem:orthogonal}
If for some $\pm x_1,\pm x_2\in S$ we have $x_1\bot x_2$, then every $\pm y\in S$ is orthogonal to at least one of the elements $\pm x_1$ or $\pm x_2$.
\end{lemma}
\begin{proof}
Take some $\pm y\in S$ with $\pm y\neq \pm x_1$. We evidently have
\begin{align*}
&p_{\pm x_1}(\pm x_2) = \pm x_2,\\
&p_{\pm x_1} (\pm y) = \pm\frac{y-x_1 \langle x_1,y\rangle}{\sqrt{1-\langle x_1,y\rangle^2}}.
\end{align*}
It follows that
$$
\langle p_{\pm x_1}(\pm x_2), p_{\pm x_1}(\pm y)\rangle
=
\pm
\left\langle x_2,\frac{y-x_1 \langle x_1,y\rangle}{\sqrt{1-\langle x_1,y\rangle^2}}\right\rangle
=
\pm\frac{\langle x_2,y\rangle}{\sqrt{1-\langle x_1,y\rangle^2}}.
$$
On the other hand, by~\eqref{eq:scalar_prod_preserved} we have
$$
\langle p_{\pm x_1}(\pm x_2), p_{\pm x_1}(\pm y)\rangle
=
\pm \langle x_2,y\rangle
.
$$
Comparing these two results, we obtain that $\langle x_1,y\rangle =0$ or $\langle x_2,y\rangle =0$.
\end{proof}

\begin{lemma}\label{lem:orthogonal_equivalence}
For two elements $\pm x,\pm y\in S$ write $\pm x\sim \pm y$ if $\langle x, y\rangle \neq 0$. Then, $\sim$ is an equivalence relation on $S$.
\end{lemma}
\begin{proof}
It is clear that $\pm x\sim \pm x$. Also,  $\pm x\sim \pm y$ if and only if $\pm y\sim \pm x$. We show that the relation $\sim$ is transitive. Let $\pm x\sim \pm y$ and $\pm y\sim \pm z$. If, by contraposition, $x$ is orthogonal to $z$, then by Lemma~\ref{lem:orthogonal} we would have $y\bot x$ or $y\bot z$, which is in both cases a contradiction. So, $x$ is not orthogonal to $z$, which means that $\pm x\sim \pm z$.
\end{proof}

By Lemma~\ref{lem:orthogonal_equivalence}, we can always decompose $S$ into pairwise orthogonal equivalence classes and analyse these separately. In the following, we assume that $S$ is irreducible, that is it consists of just one equivalence class. We shall now construct a set $T\subset \SS(H)$ (not $\P(H)$!) such that for every $\pm s\in S$ we have either $s\in T$ or $-s\in T$, but not both. Take some arbitrary $\pm s_0\in S$ and define
$$
T = \{x\in \SS(H)\colon \pm x\in S, \langle x, s_0\rangle>0\}.
$$
Note that $s_0\in T$ and $\langle x, s_0\rangle>0$ for all $x\in T$.

\begin{lemma}\label{lem:choose_positive_products}
We have $\langle x,y\rangle >0$ for all $x,y\in T$.
\end{lemma}
\begin{proof}
The claim is trivial if $x=s_0$ or $y=s_0$, so let in the following $x,y\in T\backslash \{s_0\}$.
We have
$$
p_{\pm s_0}(\pm x) = \pm\frac{x-s_0 \langle s_0,x\rangle}{\sqrt{1-\langle s_0,x\rangle^2}},
\qquad
p_{\pm s_0}(\pm y) = \pm\frac{y-s_0 \langle s_0,y\rangle}{\sqrt{1-\langle s_0,y\rangle^2}}.
$$
By~\eqref{eq:scalar_prod_preserved}, we have
\begin{align}\label{eq:tech1}
\langle p_{\pm s_0}(\pm x), p_{\pm s_0}(\pm y) \rangle
=
\pm \langle x, y\rangle.
\end{align}
On the other hand,
\begin{align}\label{eq:tech2}
\langle p_{\pm s_0}(\pm x), p_{\pm s_0}(\pm y) \rangle
=
\pm \frac{\langle x, y\rangle - \langle x,s_0 \rangle \langle y,s_0\rangle}{\sqrt{1-\langle s_0,x\rangle^2}\sqrt{1-\langle s_0,y\rangle^2}}.
\end{align}
Note that $\sqrt{1-\langle s_0,x\rangle^2}\sqrt{1-\langle s_0,y\rangle^2} < 1$ and $\langle x,s_0 \rangle \langle y,s_0\rangle>0$. Assuming by contraposition that $\langle x, y\rangle \leq  0$, we obtain
$$
\left|\frac{\langle x, y\rangle - \langle x,s_0 \rangle \langle y,s_0\rangle}{\sqrt{1-\langle s_0,x\rangle^2}\sqrt{1-\langle s_0,y\rangle^2}}\right|
=
\frac{|\langle x, y\rangle| + \langle x,s_0 \rangle \langle y,s_0\rangle}{\sqrt{1-\langle s_0,x\rangle^2}\sqrt{1-\langle s_0,y\rangle^2}}
 > |\langle x,y\rangle|,
$$
which is a contradiction.
\end{proof}

So $\langle x,y\rangle >0$ for all $x,y\in T$ and from equations~\eqref{eq:tech1} and~\eqref{eq:tech2} with $s_0$ replaced by an arbitrary $z\in T$ we get
$$
\frac{\langle x, y\rangle - \langle x,z \rangle \langle y,z\rangle}{\sqrt{1-\langle x,z\rangle^2}\sqrt{1-\langle y,z\rangle^2}}
=
\pm \langle x, y\rangle
$$
for all $x,y,z\in T$ such that $x\neq z$ and $y\neq z$.
Consider the function
$$
b(x,y) = \frac 1 {\langle x, y\rangle}, \qquad x,y\in T.
$$
Then, $b(x,x) = 1$ and $b(x,y) >1$ for $x\neq y$. The above functional equation takes the form
$$
\frac{b(x,z)b(y,z) - b(x,y)}{\sqrt{b^2(x,z) -1} \sqrt{b^2(y,z) -1}} = \pm 1
$$
for all $x,y,z\in T$ such that $x\neq z$ and $y\neq z$. Now we can introduce the function $c(x,y)$ as the only solution of
$$
b(x,y) = \frac 12 \left(c(x,y) + \frac 1 {c(x,y)}\right)
$$
with $c(x,x) = 1$ and $c(x,y)>1$ for $x\neq y$. The other solution is then $1/c(x,y)<1$.
The above functional equation takes the form
\begin{multline}\label{eq:funct_eq_c}
\left(c(x,z) + \frac 1 {c(x,z)}\right)\left(c(y,z) + \frac 1 {c(y,z)}\right) - 2 \left(c(x,y) + \frac 1 {c(x,y)}\right)
\\=
\pm  \left(c(x,z) - \frac 1 {c(x,z)}\right) \left(c(y,z) - \frac 1{c(y,z)}\right)
\end{multline}
for all $x,y,z\in T$ such that $x\neq z$ and $y\neq z$. One easily checks that this in fact holds for all $x,y,z\in T$.
If the sign on the right-hand side of~\eqref{eq:funct_eq_c} is positive, one gets after simple algebra
$$
\frac{c(y,z)}{c(x,z)} + \frac{c(x,z)}{c(y,z)} = c(x,y) + \frac 1 {c(x,y)},
$$
which implies that
$$
\frac{c(x,z)}{c(y,z)} = c(x,y)
\quad
\text{ or }
\quad
\frac{c(y,z)}{c(x,z)} = c(x,y).
$$
If the sign on the right-hand side of~\eqref{eq:funct_eq_c} is negative, then we similarly arrive at
$$
c(x,z)c(y,z) + \frac 1 {c(x,z)c(y,z)} = c(x,y) + \frac 1 {c(x,y)},
$$
which implies that
$$
c(x,z)c(y,z) = c(x,y)
\quad
\text{ or }
\quad
c(x,z) c(y,z) = \frac 1 {c(x,y)}.
$$
The latter equality is impossible if not all points $x,y,z$ are equal because $c(x,y)\geq 1$ with equality only if $x=y$.   

To summarize: For arbitrary $x,y,z\in T$, one of the three numbers $c(x,y)$, $c(y,z)$, $c(z,x)$ equals the product of the remaining two.
Introducing finally
$$
d(x,y) =\log c(x,y),
$$
we see that $d$ is a metric on $T$ that satisfies the triangle equality. By Theorem~\ref{theo:main2}, $(T,d)$ is either isometric to a subset $A$ of the real line, or to a four-point metric space from Example~\ref{example:4_points}. Observing that the scalar product $\langle x, y\rangle$ is related to $d(x,y)$ via
$$
\langle x, y\rangle = \frac 1 {b(x,y)} = \frac 2 {c(x,y) + 1/c(x,y)} = \frac {1}{\cosh d(x,y)},
$$
we arrive at the conclusion that (in the irreducible case) $S$ is isometric either to a subset of the helix from Example~\ref{ex:helix} or to some four-point set from Example~\ref{ex:four_points_projective}.

\subsection{Proof of Corollary~\ref{cor:curve}}
Let $S := \gamma(\R)= \cup_{\alpha\in I} S_\alpha$ be the decomposition given in Theorem~\ref{theo:main1}. The set $\gamma(\R)$, being a continuous image of a connected set, is connected in the topology induced from $\P(H)$.  Since the distance between any elements $\pm u\in S_\alpha$ and $\pm v\in S_\beta$ with $\alpha\neq \beta$ equals $\pi/2$, the connectedness of $\gamma(\R)$ implies that there is just one set $S_\alpha$ in the decomposition. It cannot be a four-point configuration since $\gamma(\R)$ is infinite by the injectivity of $\gamma$. So, $\gamma(\R)$ is isometric to $\{h(t)\colon t\in A\}$ for some set $A\subset \R$.

We claim that if $a_1<a_2<a_3$ are real numbers with $a_1\in A$ and $a_3\in A$, then $a_2\in A$. Indeed, assuming that $a_2\notin A$, we can represent $h(A)$ as a disjoint union of the sets $\{h(t)\colon t\in A, t<a_2\}$ and $\{h(t)\colon t\in A, t>a_2\}$. Both sets are non-empty since they contain $h(a_1)$ and $h(a_3)$, respectively, and both are open in the induced topology of $h(A)$ because $h$, being a homeomorphism between $A$ and $h(A)$, is an open map. But this is a contradiction, since $h(A)$ is isometric to $S$, which is connected. The fact that $h:\R \to \P(H)$ is a homeomorphism onto its image follows from the formula
$$
d(h(s), h(t)) =  \arccos \frac{1}{\cosh (t-s)}.
$$
It follows that $A$ is a (not necessarily open) interval. Assume, for example, that $A=[a,b)$. On the one hand, $h(A)$ is homeomorphic to $[a,b)$. On the other hand, $h(A)$ is isometric to $\gamma(\R)$. Hence, $\gamma(\R)$ is homeomorphic to $[a,b)$. That is, $\gamma$ is an injective, continuous map from $\R$ to a metric space homeomorphic to $[a,b)$. However, such map does not exist. Indeed, if $\gamma_*:\R\to [a,b)$ is injective and continuous, then it must be strictly monotone. But then $\gamma_*(\R)$ is an open interval, a contradiction.

\subsection{Proof of Corollary~\ref{cor:gaussian}}
For every point $t\in T$ with $Z(t)=0$ a.s.\ we can define a class $T_\alpha= \{t\}$ and put $a(t)=0$. In the following, let $\Var Z(t)>0$ for all $t\in T$. We may even assume that $\Var Z(t) = 1$, otherwise replace $Z(t)$ by $Z(t)/\sqrt{\Var Z(t)}$. 
Declare two points $t_1,t_2\in T$ to be equivalent if $\Cov(Z(t_1), Z(t_2)) = \pm 1$. After selecting one representative from each equivalence class and discarding the remaining elements, we may assume that $|\Cov(Z(s),Z(t))|<1$ for all $s\neq t$. (Note that in the statement of the corollary, $\psi_\alpha$ is not required to be injective and, in fact, we choose  it to be constant on equivalence classes).

It is known that the conditional law of the process $(Z(t))_{t\in T}$ given that $Z(s_0)=0$ is the same as the law of the process $(Z(t) - \Cov(Z(t),Z(s_0))Z(s_0))_{t\in T}$. On the other hand, it is the same as the law of the process $(\varphi(t;s_0)Z(t))_{t\in T}$. Comparing the variances, we arrive at $1 - \Cov^2(Z(t),Z(s_0))  = \varphi^2(t; s_0)$, so that $\varphi(t; s_0)\neq 0$ for $t\neq s_0$. Standardizing both processes, we obtain
\begin{equation}\label{eq:cond_covar}
\left(\frac{Z(t)-\Cov(Z(t),Z(s_0))Z(s_0)}{\sqrt{1-\Cov^2(Z(t),Z(s_0))}}\right)_{t\in T\backslash\{s_0\}} \eqfdd (Z(t))_{t\in T\backslash\{s_0\}}.
\end{equation}
Now let $H$ be the $L^2$-space of the probability space on which the process $Z$ is defined and consider the set
$
S:= \{ \pm Z(t) \colon t\in T\}\subset \P(H). 
$
Recall from~\eqref{eq:proj_def} that 
$$
p_{\pm Z(s_0)}(\pm Z(t)) = \pm \frac{Z(t)-\Cov(Z(t),Z(s_0))Z(s_0)}{\sqrt{1-\Cov^2(Z(t),Z(s_0))}},
\qquad t\in T\backslash\{s_0\}. 
$$  
In the Hilbert space notation, the equality of the covariances of the processes in~\eqref{eq:cond_covar} implies that
$$
|\langle p_{\pm Z(s_0)}(\pm Z(x)), p_{\pm Z(s_0)}(\pm Z(y)) \rangle| = |\langle Z(x), Z(y)\rangle|
$$
for all $x,y\in T\backslash\{s_0\}$, which means that $S$ satisfies the condition of Definition~\ref{def:main}.  Theorem~\ref{theo:main1} yields a decomposition $T = \cup_{\alpha\in I} T_\alpha$ such that the sets $\{\pm Z(t)\colon t\in T_\alpha\}\subset \P(H)$ are mutually orthogonal, which means that the Gaussian  processes $(Z(t))_{t\in T_\alpha}$ are mutually independent. Moreover, for each $\alpha\in I$, the set $\{\pm Z(t)\colon t\in T_\alpha\}$ is isometric to $h(A)$ for some set $A\subset \R$ or to a four-point configuration from Example~\ref{ex:four_points_projective}. For concreteness, let us consider the former case. The existence of the isometry means that there is a bijection $\psi_\alpha: T_\alpha \to A$ such that
$$
\arccos |\Cov(Z(s), Z(t))| = \arccos \frac 1 {\cosh (\psi_\alpha(t)-\psi_\alpha(s))}, \qquad s,t\in T_\alpha.
$$
Hence, $|\Cov(Z(s), Z(t))| = \Cov(X(\psi_\alpha(s)),X(\psi_\alpha(t)))$. By Lemma~\ref{lem:choose_positive_products}, there is a function $a: T_\alpha \to \{-1,1\}$ such that $\Cov(a(s) Z(s), a(t) Z(t))>0$ for all $s,t\in T_\alpha$, which implies that
$$
\Cov(a(s)Z(s), a(t)Z(t)) = \Cov(X(\psi_\alpha(s)),X(\psi_\alpha(t))).
$$
Hence, the process $(Z(t))_{t\in T_\alpha}$ has the same law as $(a(t)X(\psi_\alpha(t)))_{t\in T_\alpha}$.


\section{Proof of Theorem~\ref{theo:main2}}\label{proof:main_2}

Consider a metric space $(E,d)$  satisfying the triangle equality. It is clear that if $E$ has $\leq 3$ points, then it can be embedded into $\R$ isometrically.

Let the number of points in $E$ be equal to $4$.  Without loss of generality, let the diameter of this space be $1$. Otherwise, we can rescale the distances. Let $0$ and $1$ be the points in $E$ with $d(0,1)=1$ and denote the remaining two points by $X$ and $Y$. Let
$$
d(0,X) = x, \qquad d(1,Y)=y, \qquad d(X,Y) = d.
$$
We have $x<1$, $y<1$ and $d\leq 1$. Indeed, $x=1$, would imply that the triangle $01X$ could satisfy the triangle equality only if $X=1$. Similarly $y=1$ is not possible. With the above notation, from the triangles $01X$ and $01Y$ we have
$$
d(X,1) = 1-x, \qquad d(Y,0) = 1-y.
$$
We consider the triangles $0XY$ and $1XY$. There are $9$ cases.

\vspace*{2mm}
\noindent
\textit{Case 1:} $x+ (1-y) = d$ and $(1-x)+y = d$. It follows that $x=y$ and hence $d=1$. Thus, our metric space is isometric to a space from Example~\ref{example:4_points}.

\vspace*{2mm}
\noindent
\textit{Case 2:} $x+d= 1-y$ and $(1-x)+y =d$. It follows that $y=0$, a contradiction.

\vspace*{2mm}
\noindent
\textit{Case 3:} $x= (1-y) +d$ and $(1-x)+y = d$. It follows that $x=1$ and so $d(X,1) = 1-x = 0$, a contradiction.

\vspace*{2mm}
\noindent
\textit{Case 4:} $x+(1-y)=d$ and $(1-x)+d =y$. It follows that $y=1$, hence $d(Y,0) = 1-y=0$, a contradiction.

\vspace*{2mm}
\noindent
\textit{Case 5:} $x+d=1-y$ and $(1-x)+d = y$. It follows that $d=0$, a contradiction.

\vspace*{2mm}
\noindent
\textit{Case 6:} $x = (1-y) +d$ and $(1-x)+d = y$. Both equations are equivalent to $x+y=d+1$. Then, the map
$$
\varphi(0)=0, \quad \varphi(1)=1, \quad \varphi(X) = x, \quad \varphi(Y) = 1-y
$$
defines an isometric embedding of $E$ into $\R$.

\vspace*{2mm}
\noindent
\textit{Case 7:} $x+(1-y)=d$ and $1-x = d+y$. It follows that $x=0$, a contradiction.

\vspace*{2mm}
\noindent
\textit{Case 8:} $x+d= 1-y$ and $1-x=d+y$. Both conditions are equivalent to $x+y+d=1$. The map
$$
\varphi(0) =0,\quad \varphi(1) = 1,\quad \varphi(X)=x, \quad \varphi(Y)= 1-y
$$
defines an isometric embedding of $E$ into $\R$.

\vspace*{2mm}
\noindent
\textit{Case 9:} $x=(1-y)+d$ and $1 - x = d + y$. It follows that $d=0$, a contradiction.

This completes the proof in the case of $4$ points.

Let now $E$ be a metric space consisting of  exactly $5$ points. Again assume that the diameter is $1$ and that the points are $0,1,X,Y,Z$ with
$$
d(0,1) = 1.
$$
Consider the quadruple $\{0,1,X,Y\}$. It is either ``classical'' (that is, it can be isometrically embedded into $\R$) or it is non-classical (that is, it is isometric to the space in Example~\ref{example:4_points}). Our aim is to show that the latter case cannot occur.

\vspace*{2mm}
\noindent
\textsc{Assumption:} The quadruple $\{0,1,X,Y\}$ is non-classical, namely
$$
d(0,1) = d(X,Y) = 1, \quad d(0,X) = d(1,Y) = x, \quad d(0,Y)= d(1,X) = 1-x.
$$
Now let us look at the quadruple $\{0,1,X,Z\}$ and consider two cases.

\vspace*{2mm}
\noindent
\textsc{Case 1:} The quadruple $\{0,1,X,Z\}$ is classical. We may then identify $X$ and $Z$ with two points $x$ and $z$ in the interval $(0,1)$. (Recall that the diameter of $E$ is $1$). At the moment, we don't know which of the numbers, $x$ or $z$, is larger.     So, we have the points $0,1,x,z$ on the real line and one additional point $Y$ outside with
$$
d(Y,0) = 1-x, \quad d(Y,x) = 1, \quad d(Y,z)=:u, \quad d(Y,1) = x.
$$
The following triangles satisfy the triangle equality:
\begin{itemize}
\item $0YZ$ with side lengths $1-x,u,z$.
\item $XYZ$ with side lengths $1,u, |z-x|$.
\item $YZ1$ with side lengths $u,x, 1-z$.
\end{itemize}
Since the diameter of our space is $1$, we get from the triangle $XYZ$ that
$$
u = 1- |z-x|.
$$

\vspace*{2mm}
\noindent
\textsc{Subcase 1a:} $x<z$. We have two triangles which satisfy the  triangle equality:
\begin{itemize}
\item $0YZ$ with side lengths $1-x,1-z+x,z$.
\item $YZ1$ with side lengths $1-z+x, x, 1-z$.
\end{itemize}
For $YZ1$, the triangle equality is fulfilled, so we are left with $0YZ$. If $1-x = (1-z+x)+z$, then $x=0$ and hence $X=0$, a contradiction. If $1-z+x = (1-x)+z$, then $x=z$ and hence $X=Z$, a contradiction. Finally, if $z= (1-x) + (1-z+x)$, then $z=1$ and hence $Z=1$, a contradiction.

\vspace*{2mm}
\noindent
\textsc{Subcase 1a:} $x>z$.
We have two triangles satisfying the  triangle equality:
\begin{itemize}
\item $0YZ$ with side lengths $1-x,1-x+z,z$.
\item $YZ1$ with lengths $1-x+z, x, 1-z$.
\end{itemize}
In $0YZ$, the triangle equality holds trivially, and we are left with the triangle $YZ1$.
If $1-x+z = x+(1-z)$, then $x=z$, a contradiction.
If $x= (1-x+z) + (1-z)$, then $x=1$, a contradiction. Finally, if $1-z = (1-x+z) + x$, then $z=0$, again a contradiction.

We arrive at the conclusion that case $1$ is not possible.

\vspace*{2mm}
\noindent
\textsc{Case 2:} The quadruple $\{0,1,X,Z\}$ is non-classical, see Example~\ref{example:4_points}. It follows that $d(X,Z)= 1$. Consider now the quadruple $\{0,1, Y,Z\}$. Since, as we argued above, Case 1 leads to a contradiction, this quadruple must be non-classical, too. Hence, $d(Y,Z)=1$. Recalling that $d(X,Y) = 1$, we arrive at the contradiction because the triangle equality does not hold for the triangle $XYZ$.

\vspace*{2mm}
\noindent
So, our assumption was wrong and the quadruple $\{0,1,X,Y\}$ is classical. Similarly, the quadruples $\{0,1,Y,Z\}$ and $\{0,1,Z,X\}$ are classical. Without loss of generality, let $0<d(0,X) < d(0,Y) <d(0,Z)<1$. Consider the map $\varphi:E\to [0,1]$ with
$$
\varphi(0)= 0,
\quad
\varphi(1)= 1,
\quad
\varphi(X) = d(0,X),
\quad
\varphi(Y) = d(0,Y),
\quad
\varphi(Z) = d(0,Z).
$$
We prove that it is an isometry. Since the diameter of $E$ is $1$, we have
$$
d(X,1) = 1 - d(0,X) = |\varphi(X) - \varphi(1)|,
$$
and similarly for $d(Y,1)$ and $d(Z,1)$. To complete the proof that $\varphi$ is an isometry, we have to show that $|\varphi(X) - \varphi(Y)| = d(X,Y)$ (for the pairs $Y,Z$ and $X,Z$ the proof is analogous). But this claim follows from the fact that the quadruple $\{0,1,X,Y\}$ is classical.

\vspace*{2mm}
\noindent
Let finally $E$ have an arbitrary cardinality $\geq 5$. We have shown above that every quadruple in $E$ (and, in fact, every five-point set) admits an isometric embedding into $\R$. We shall now define an isometric embedding $\varphi: E\to\R$. Take arbitrary points $A,B\in E$ with $A\neq B$ and define $\varphi(A)= 0$, $\varphi(B) = d(A,B)$. Take one more point $C\in E$. There is a unique isometry $\varphi_{C}:\{A,B,C\}\to\R$ satisfying the conditions $\varphi_C(A) = 0$, $\varphi_C(B)= d(A,B)$. We now claim that  $\varphi(C):= \varphi_C(C)$ defines an isometric embedding of $E$ into $\R$. Indeed, let $D_1,D_2\in E$ be two points. There is an isometry  $\psi: \{A,B,D_1,D_2\}\to \R$. Moreover, if we impose the conditions $\psi(A)=0$, $\psi(B)=d(A,B)$, it becomes uniquely defined. Since the restrictions of $\psi$ to the triangles $ABD_1$ and $ABD_2$ are isometries, it follows that $\varphi(D_i) = \varphi_{D_i}(D_i)= \psi(D_i)$ for $i=1,2$. But then $d(\varphi(D_1),\varphi(D_2)) = d(\psi(D_1),\psi(D_2)) = d(D_1,D_2)$ because $\psi$ is isometric.

\bibliography{helix_bib}
\bibliographystyle{plainnat}




\end{document}